\documentclass[11pt, reqno]{amsart}
\usepackage{graphicx}
\usepackage{amsmath}
\usepackage{amsthm}
\usepackage{amssymb,bbm}%
\usepackage{hyperref}
\usepackage[numbers, square]{natbib}
\usepackage{color}

\newcommand{\eee}{{\rm e}}
\newcommand{\ii}{{\rm{i}}}

  \evensidemargin
\oddsidemargin
\DeclareMathOperator{\Ima}{Im}
\DeclareMathOperator{\Ker}{Ker}

\theoremstyle{plain}
\newtheorem{theorem}{Theorem}[section]
\newtheorem{lemma}[theorem]{Lemma}
\newtheorem{corollary}[theorem]{Corollary}
\newtheorem{proposition}[theorem]{Proposition}

\theoremstyle{definition}

\theoremstyle{remark}
\newtheorem{remark}[theorem]{Remark}

\makeindex
\makeglossary

\newcommand{\bB}{\mathbb{B}}

\newcommand{\cF}{\mathcal{F}}

\newcommand{\cQ}{\mathcal{Q}}
\newcommand{\cP}{\mathcal{P}}

\newcommand{\E}{\mathbb E\,}
\newcommand{\R}{\mathbb{R}}
\newcommand{\N}{\mathbb{N}}
\newcommand{\C}{\mathbb{C}}

\renewcommand{\P}{\mathbb{P}}

\newcommand{\Int}{\mathop{\mathrm{int}}\nolimits}

\newcommand{\relint}{\mathop{\mathrm{relint}}\nolimits}

\newcommand{\Vol}{\mathop{\mathrm{Vol}}\nolimits}

\newcommand{\conv}{\mathop{\mathrm{conv}}\nolimits}

\newcommand{\pos}{\mathop{\mathrm{pos}}\nolimits}
\newcommand{\lin}{\mathop{\mathrm{lin}}\nolimits}
\newcommand{\aff}{\mathop{\mathrm{aff}}\nolimits}

\newcommand{\rank}{\mathop{\mathrm{rank}}\nolimits}
\newcommand{\codim}{\mathop{\mathrm{codim}}\nolimits}

\newcommand{\eqdistr}{\stackrel{d}{=}}

\newcommand{\ind}{\mathbbm{1}}

\newcommand{\dd}{{\rm d}}

\newcommand{\stirling}[2]{\genfrac{[}{]}{0pt}{}{#1}{#2}}
\newcommand{\stirlingsec}[2]{\genfrac{\{}{\}}{0pt}{}{#1}{#2}}

\begin{document}

\author{Thomas Godland}
\address{Institut f\"ur Mathematische Stochastik,
Westf\"alische Wilhelms-Universit\"at M\"unster,
Orl\'eans-Ring 10,
48149 M\"unster, Germany}
\email{t\_godl01@uni-muenster.de}

\author{Zakhar Kabluchko}
\address{Zakhar Kabluchko: Institut f\"ur Mathematische Stochastik,
Westf\"alische {Wilhelms-Uni\-ver\-sit\"at} M\"unster,
Orl\'eans--Ring 10,
48149 M\"unster, Germany}
\email{zakhar.kabluchko@uni-muenster.de}

\author[D.~Zaporozhets]{Dmitry Zaporozhets}
\address{Dmitry Zaporozhets: St.~Petersburg Department of Steklov Institute of Mathematics, Fontanka~27, 191011 St.~Petersburg, Russia}
\email{zap1979@gmail.com}

\title[Angle sums of random polytopes]{Angle sums of random polytopes}

\thanks{
The work of DZ and ZK has been supported by RFBR and
DFG according to the research project № 20-51-12004.
TG and ZK acknowledge support by the German Research Foundation under Germany's Excellence Strategy  EXC 2044 -- 390685587, Mathematics M\"unster: Dynamics-Geometry-Structure.}

\begin{abstract}
For two families of random polytopes we compute explicitly the expected sums of the conic intrinsic volumes and the Grassmann angles at all faces of any given dimension of the polytope under consideration. As special cases, we compute the expected  sums of internal and external angles at all faces of any fixed dimension.  The first family are the Gaussian polytopes defined as convex hulls of i.i.d.\ samples from a non-degenerate Gaussian distribution in $\mathbb R^d$.  The second family are convex hulls of random walks with exchangeable increments satisfying certain mild general position assumption. The expected sums are expressed in terms of the angles of the regular simplices and the Stirling numbers, respectively. There are non-trivial analogies between these two settings. Further, we compute the angle sums for Gaussian projections of arbitrary polyhedral sets, of which the Gaussian polytopes are a special case.   Also, we show that the expected Grassmann angle sums of a random polytope with a rotationally invariant law are invariant under affine transformations. Of independent interest may be also results on the faces of linear images of polyhedral sets. These results are well known but it seems that no detailed proofs can be found in the existing literature.
\end{abstract}


\maketitle

\section{Introduction}
\subsection{Angles and face numbers}
For a convex  polytope $P\subset\R^d$ denote by $\cF(P)$ the set of its faces including $P$ itself. The classical Euler relation (see, e.g.,~\cite[Chapter~8]{bG03}) states that for every polytope $P$,
\begin{align}\label{916}
    \sum_{F\in\cF(P)}(-1)^{\dim F}=1.
\end{align}
A similar, although slightly less known result, exists for the internal solid angles of $P$. Let $\beta(F,P)$ denote the internal solid angle of $P$ at the face $F$. It can be defined as
$$
\beta(F, P)
:=
\lim_{r\downarrow 0} \frac{\Vol(\bB_r(z)\cap P) }{\Vol(\bB_r(z))},
$$
where $\Vol$ denotes the Lebesgue measure in $\R^d$, $\bB_r(z)$ is the $d$-dimensional ball with radius $r>0$ centered at $z$, and $z$ is any point in $F$ not belonging to a face of smaller dimension.
Then the following Gram--Euler relation holds:
\begin{align}\label{934}
    \sum_{F\in\cF(P)}(-1)^{\dim F}\beta(F,P)=0,
\end{align}
see~\cite{jG74} for $d=3$ and~\cite[\S14.1]{bG03} for arbitrary dimension.
For $d=2$, this relation reduces to a theorem from plane geometry stating that the angle-sum of any $n$-gon equals  $(n-2)\pi$.

Perles and Shephard~\cite[\S2]{PS67} found an elegant derivation of~\eqref{934} from~\eqref{916}. To this end,  they considered a random orthogonal projection $\Pi_{d-1}P$ of $P$ onto  a random $(d-1)$-dimensional hyperplane whose normal vector is uniformly distributed on the unit sphere in $\R^d$. They observed~\cite[Eq.~(8)]{PS67} that for all $j\in\{0,\ldots, d-1\}$,
\begin{align}\label{eq:sum_beta}
    \sum_{F\in\cF_j(P)} \beta(F,P) = \frac12 f_j(P)- \frac12 \E  f_j(\Pi_{d-1} P),
\end{align}
where $\cF_j(P)$ denotes the set of all $j$-dimensional faces of a polytope $P$, and  $f_j(P)= |\cF_j(P)|$ is their number.
Multiplying~\eqref{eq:sum_beta} by $(-1)^j$, taking the sum over all dimensions $j\in \{0,\ldots,d-1\}$, and making use of the Euler relation~\eqref{916} for $P$ and $\Pi_{d-1}P$, Perles and Shephard derived~\eqref{934}.

Shortly after that, Gr\"unbaum~\cite{bG68} generalized this approach to the so-called Grassmann angles (to be defined in Section~\ref{559}) and proved various linear relations for these angles.  Relating expected face numbers of the random projections  to the angles of the polytope  is a crucial step in the work of Affentranger and Schneider~\cite{AS92}. Later, similar ideas were also used in \cite{FK09,kabluchko_angles,beta_polytopes}.

\subsection{Outline of the paper}
Our goal is to apply the idea of Perles and Shephard to compute the expected sums of the angles for \emph{random} convex polytopes. We will consider two basic models: the \emph{Gaussian polytopes} (and, more generally, Gaussian projections of arbitrary polyhedral sets) and the \emph{convex hulls of random walks}. A detailed description of these models will be given in Section~\ref{2320} and Section~\ref{subsec:gauss_proj}. Necessary preliminaries are collected in Section~\ref{559}.  The expected number of the $j$-faces is known for both models, see~\cite{AS92} (combined with~\cite{BV94}) for the Gaussian polytopes and~\cite{KVZ17} for the convex hulls of random walks. Moreover, these models possess the common important property: the random projection of the polytope from the model has the same distribution as the same model of lower dimension. This makes it possible to use the approach of Perles and Shephard to compute the expected sums of the angles for these models. Furthermore, like in~\cite{bG68}, we will also generalize these results to sums of Grassmann angles which include both internal and external angles as special cases. The main results and their proofs are collected in Sections~\ref{1025} and~\ref{631}.

\section{Convex cones and Grassmann angles}\label{559}
In this section we collect some necessary definitions from stochastic and convex geometry. The reader may skip this section  and return to it when necessary.

\subsection{Notation}
For a set $M\subset\R^d$ denote by $\lin  M$ (respectively, $\aff M)$ its linear (respectively, affine) hull, that is, the minimal linear (respectively, affine) subspace containing $M$.
Equivalently,  $\lin  M$ (respectively, $\aff M)$ is the set of all linear (respectively, affine) combinations of elements of $M$. The interior of $M$ will be denoted by $\Int M$.
We write $\relint M$ for the relative interior of $M$ which is the interior of $M$ taken with respect to its affine hull $\aff M$. The dimension of  a convex set $M$, denoted by $\dim M$, is  the dimension of $\aff M$.

For an arbitrary   set  $M\subset \R^d$ let $\pos  M$ denote its  \emph{positive} (or \emph{conic}) \emph{hull}:
\[
\pos   M: =\Big\{\sum_{i=1}^m\lambda_i t_i:\, m\in \N,\,  t_1, \ldots,t_m\in M,\, \lambda_1,\ldots,\lambda_m\geq 0\Big\}.
\]

\subsection{Grassmann angles}
A set $C\subset\R^d$ is called a \emph{polyhedral cone} if it can be represented as a positive hull of finitely many vectors. Equivalently, a polyhedral cone is an intersection of finitely many half-spaces whose boundaries pass through the origin.
The \emph{solid angle} of a polyhedral cone $C\subset\R^d$ is defined as
\begin{equation}\label{933}
\alpha(C):=\P[Z\in C],
\end{equation}
where $Z$  is uniformly distributed on the centered unit sphere in the linear hull $\lin C$. The maximal possible value of the solid angle in this normalization is $\alpha(C)=1$ and attained if $C$ is a linear subspace. If the dimension of $C$ is $d$ but  $C\ne\R^d$, then $\P[Z\in C,-Z\in C]=0$ and denoting the random line passing through $Z$ and $-Z$ by $W_1$, we obtain that~\eqref{933} is equivalent to
\begin{equation}\label{1304}
\alpha(C) = \frac12\P[W_1\cap C\ne\{0\}].
\end{equation}

This definition of the solid angle can be generalized as follows.  Fix some $k\in \{0,\ldots,d\}$.  Let $W_{d-k}$ be a random $(d-k)$-dimensional linear subspace having the uniform distribution on the Grassmannn manifold of all such subspaces.
Following Gr\"unbaum~\cite{bG68}  define (with the inverse index order) the $k$-th \emph{Grassmann angle} of $C$ as the probability that  $C$ is intersected by the random, uniform $(d-k)$-plane $W_{d-k}$ non-trivially:
\begin{equation}\label{1138}
\gamma_k(C):=\P[W_{d-k}\cap C\ne\{0\}], \quad k\in \{0,\ldots,d\}.
\end{equation}
For example, taking $k=d-1$  and assuming that the dimension of $C$ is $d$, we have
\begin{align}\label{739}
  \alpha(C)=\frac 1 2\gamma_{d-1}(C)+\frac 1 2\ind[C=\R^d].
\end{align}
It follows from~\eqref{1138} that for any convex cone $C\subset\R^d$ with $C\neq \{0\}$,
$$
1= \gamma_0(C) \geq \gamma_1(C) \geq \ldots \geq \gamma_{d}(C) = 0.
$$
The \textit{lineality space} of a polyhedral cone $C$, defined as $C\cap (-C)$, is the maximal linear subspace contained in $C$. If the lineality space of $C$ has dimension $j\in \{0,\ldots,d-1\}$ and $C$ is not a linear subspace, then~\eqref{1138} even implies that
\begin{equation}\label{eq:grassmann_angles_lineality}
1= \gamma_0(C) =\ldots =  \gamma_{j}(C) \geq \gamma_{j+1}(C)  \geq \ldots \geq \gamma_{d}(C) = 0.
\end{equation}
On the other hand, it follows directly from~\eqref{1138} that for a $j$-dimensional linear subspace  $L_j\subset\R^d$ with $j\in \{0,\ldots,d\}$ we have
\begin{equation}\label{2256}
\gamma_k(L_j)=
\begin{cases}
1,\quad \text{ if } 0\leq k\leq j-1,\\
0,\quad \text{ if } j\leq k\leq d.
\end{cases}
\end{equation}
If $C$ is not a linear subspace, then the quantity $\frac 12 \gamma_k(C)$ is also known as the \textit{$k$-th  conical quermassintegral $U_k(C)$} of $C$; see~\cite[Eqs.~(1)--(4)]{HugSchneider2016} or as the  \emph{half-tail functional} $h_{k+1}(C)$ defined in~\cite{amelunxen_edge}.

It was shown in~\cite[Eq.~(2.5)]{bG68}) that, as with the classical intrinsic volumes, the Grassmann angles do not depend on dimension of the ambient space: If we embed $C$ in $\R^N$ with $N\geq d$, the result will be the same. In particular, it is convenient to define
\[
\gamma_N(C):=0\quad\text{for all} \quad N\geq\dim C.
\]


\subsection{Angles of polyhedral sets}
A \textit{polyhedral set} is an intersection of finitely many closed half-spaces (whose boundaries need not pass through the origin).
If a polyhedral set is bounded, it is a polytope. Polyhedral cones are also special cases of polyhedral sets.
Denote by $\cF_j(P)$ the set of $j$-dimensional faces of a polyhedral set $P\subset \R^d$. The \textit{tangent cone} at a face $F\in \cF_j(P)$ is defined by
\begin{equation}\label{eq:def_tangent_cone}
T_F(P) = \{v\in\R^d\colon f_0 +\varepsilon v \in P \text{ for some } \varepsilon>0\},
\end{equation}
where $f_0$ is any point in the relative interior of $F$.  
The \textit{normal cone} at the face $F\in \cF_j(P)$ is defined as the polar of the tangent one, that is
\begin{equation}\label{eq:def_normal_cone}
N_F(P) = T_F^\circ (P) = \{w\in\R^d \colon \langle w, u\rangle\leq 0 \text{ for all } u\in T_F(P)\}.
\end{equation}
The \textit{internal angle} of $P$ at $F$ is defined as the solid angle of its tangent cone:
$$
\beta(F,P) := \alpha(T_F(P)).
$$
The \textit{external angle} of $P$ at $F$ is the solid angle of its normal cone
$$
\gamma(F,P) := \alpha (N_F(P)).
$$

\section{Two models of random polytopes}\label{2320}
\subsection{Gaussian polytopes}\label{2320a}
Let $X_1,\ldots, X_n$ be independent $d$-dimensional standard  Gaussian random vectors. Their convex hull
\begin{align}\label{2224}
    \cP_{n,d}:=\conv(X_1,\ldots,X_n)
\end{align}
is called the Gaussian polytope. Most of the time, it will be convenient to impose the assumption  $n\geq d+1$ which guarantees that $\cP_{n,d}$ has full dimension $d$ a.s. Fix some $j\in\{0,\ldots,d-1\}$. An exact formula for the expected number of $j$-dimensional faces of $\cP_{n,d}$ can be obtained by combining the results of~\citet{AS92} and~\citet{BV94}. To state this formula, we need to introduce some notation. Let $e_1,\ldots,e_n$ be the standard orthonormal basis in $\R^n$. The internal, respectively, external,  angle sums  of the regular $n$-vertex simplex $\Delta_n:=\conv(e_1,\ldots,e_n)$ at its $k$-vertex faces are denoted by
$$
\sigma\stirlingsec{n}{k},
\;\;\;
\text{respectively,}
\;\;\;
\sigma\stirling{n}{k}.
$$
This notation is intentionally chosen to resemble the standard notation for Stirling numbers~\cite[\S6.1]{Graham1994}; the analogy between these notions will be discussed below. Since the number of $k$-vertex faces of $\Delta_n$ equals $\binom nk$ and since the angles at all such faces are equal, we can choose one $k$-vertex face, say $\Delta_k:=\conv(e_1,\ldots,e_k)$, and define
$$
\sigma\stirlingsec{n}{k}:= \binom nk \cdot \alpha (T_{\Delta_k} (\Delta_n)),
\qquad
\sigma\stirling{n}{k}:=  \binom nk \cdot \alpha (N_{\Delta_k} (\Delta_n)),
$$
for all $n\in \N$ and all $k\in \{1,\ldots,n\}$. Here, $T_{\Delta_k} (\Delta_n)$,
respectively $N_{\Delta_k} (\Delta_n)$, denotes that tangent (respectively, normal) cone of $\Delta_n$ at $\Delta_k$,  while $\alpha(C)$ is the solid angle of a cone $C$; see Section~\ref{559}. It is convenient to extend the above definition by putting
$$
\sigma\stirlingsec{n}{k}
:=
\sigma\stirling{n}{k}
:=
0,
$$
for all $n\in \N$ and all $k\notin \{1,\ldots,n\}$.    With this notation, the formula of~\citet{AS92} (taking into account also the observation of~\citet{BV94}) takes the form
\begin{equation}\label{eq:E_f_k_P_n_d}
\E f_j(\cP_{n,d})
=
2\sum_{l=0}^{\infty}  \sigma\stirling{n}{d-2l} \sigma\stirlingsec{d-2l}{j+1},
\end{equation}
for all $j\in \{0,\ldots,d-1\}$.
In fact, \citet{AS92} proved the same formula for the expected number of $j$-dimensional faces of the projection of the simplex $\conv(e_1,\ldots,e_n)$ onto a uniform, random $d$-dimensional subspace in $\R^n$. Then, \citet{BV94} argued that this expected number of faces is the same as for the Gaussian polytope.

Explicit formulas for $\sigma\stirlingsec{n}{k}$ and  $\sigma\stirling{n}{k}$ are available; see~\cite{KZ17a} for a review of this topic. For example, it is known that
\begin{align}
\sigma\stirlingsec{n}{k}
&=
\binom nk \cdot \frac 1 {\sqrt {2\pi}} \int_{0}^{\infty} \left(\Phi^{n-k} \left( \frac{\ii x}{\sqrt n}\right) + \Phi^{n-k} \left(- \frac{\ii x}{\sqrt n}\right)\right) \eee^{-x^2/2} \dd x
,\label{eq:regular_simpl_internal}\\
\sigma\stirling{n}{k}
&=
\binom nk \cdot \frac 1 {\sqrt {2\pi}} \int_{0}^{\infty} \left(\Phi^{n-k} \left(\frac{x}{\sqrt k}\right) + \Phi^{n-k} \left(- \frac{x}{\sqrt k}\right)\right) \eee^{-x^2/2} \dd x,
\label{eq:regular_simpl_external}
\end{align}
where $\ii = \sqrt {-1}$, and $\Phi$ denotes the distribution function of the standard normal law.
It is known that $\Phi$ admits an analytic continuation to the entire complex plane, namely
$$
\Phi(z) = \frac 12  + \frac 1 {\sqrt{2\pi}} \sum_{n =0}^{\infty}  \frac{(-1)^n }{(2n+1) 2^n n!} z^{2n+1}, \qquad z\in \C.
$$
In the above formulas for the angle sums, we need the values of $\Phi$ on the real and imaginary axes only, namely
\begin{equation}\label{eq:Phi_def}
\Phi(z) = \frac{1}{\sqrt{2\pi}} \int_{-\infty}^z e^{-t^2/2} \dd t,
\quad
\Phi( \ii z) = \frac 12 + \frac i{\sqrt{2\pi}} \int_0^{z} e^{t^2/2} \dd t,
\quad z\in\R.
\end{equation}

\subsection{Convex hulls of random walks}\label{2320b}
Let $\xi_1,\ldots,\xi_n$ be (possibly dependent) random $d$-dimensional vectors with partial sums
$$
S_i = \xi_1 + \ldots + \xi_i,\quad  1\leq i\leq n,\quad  S_0=0.
$$
The sequence $S_0,S_1,\ldots,S_n$ will be referred to as a \emph{random walk}.
Consider its \emph{convex hull} 
\begin{align}\label{2225}
\cQ_{n,d} &:= \conv(S_0,S_1,\ldots, S_n).
\end{align}
We impose the following assumptions on the joint distribution of  the increments. 
\begin{enumerate}
\item[$(\text{Ex})$] \textit{Exchangeability:} For every permutation $\sigma$ of the set $\{1,\ldots,n\}$, we have the distributional equality
$$
(\xi_{\sigma(1)},\ldots,  \xi_{\sigma(n)}) \eqdistr (\xi_1,\ldots,\xi_n).
$$
\item[$(\text{GP})$] \textit{General position:}
For every $1\leq i_1 < \ldots < i_d\leq n$, the probability that the vectors $S_{i_1}, \ldots,S_{i_d}$ are linearly dependent is $0$.
\end{enumerate}

Under these assumptions, it was shown in~\cite{KVZ17} that for all  $j\in\{0,\ldots,d-1\}$,
\begin{align}\label{800}
\E f_j(\cQ_{n,d})= \frac{2\cdot j!}{n!} \sum_{l=0}^{\infty}\stirling{n+1}{d-2l}  \stirlingsec{d-2l}{j+1}.
\end{align}

The right-hand side  contains the (signless) \emph{Stirling numbers of the first kind} $\stirling{n}{m}$ and the \emph{Stirling numbers of the second kind} $\stirlingsec{n}{m}$, which are defined as the number of permutations of an $n$-element set with exactly $m$ cycles and the number of partitions of an $n$-element set into $m$ non-empty subsets, respectively, for $n\in\N$ and $m \in \{1,\ldots,n\}$. For $n\in\N$ and $m\notin \{1,\ldots,n\}$ one defines the Stirling numbers to be $0$, so that~\eqref{800} and all similar formulas contain a finite number of non-vanishing terms only.   For the basic properties of the Stirling numbers, we refer to~\cite[\S6.1]{Graham1994}.  The exponential generating functions of the Stirling numbers are given by
\begin{equation}\label{eq:stirling_def}
\sum_{n=m}^{\infty} \stirling{n}{m}\frac{t^n}{n!} = \frac 1 {m!} \left(\log \frac 1 {1-t}\right)^m,
\quad
\sum_{n=m}^{\infty} \stirlingsec{n}{m}\frac{t^n}{n!} = \frac 1 {m!} (e^{t}-1)^m.
\end{equation}
With the convention $\stirling{0}{0} = \stirlingsec{0}{0} =1$, the two-variable generating functions are given by
\begin{align}\label{eq:stirling_def_2_variables}
\sum_{m=0}^\infty\sum_{n=m}^{\infty} \stirling{n}{m}\frac{t^n}{n!}y^m=(1-t)^{-y}, \quad\sum_{m=0}^\infty\sum_{n=m}^{\infty} \stirlingsec{n}{m}\frac{t^n}{n!}y^m=e^{(e^t-1)y}.
\end{align}

\section{Main results}\label{1025}
\subsection{Expected sums of Grassmann angles}
Our main results are the following two theorems in which  we compute the expected  sums of the Grassmann angles at the faces of any fixed dimension for each of the random polytopes $\cP_{n,d}$ and $\cQ_{n,d}$ defined in Section~\ref{2320}.
\begin{theorem}\label{2219}
Fix some $d\in\N$ and $n\geq d+1$. Then, for every
$j\in \{0,\ldots,d-1\}$ and $k\in \{0,\ldots,d\}$
the expected sum of the $k$-th Grassmann angles at the $j$-dimensional faces of $\cP_{n,d}$ equals
\begin{equation}\label{eq:theo_sum_grassmann_P}
\E \sum_{F\in\cF_j(\cP_{n,d})}\gamma_k(T_F(\cP_{n,d}))
=
2\sum_{l=0}^{\infty}  \sigma\stirling{n}{d-2l}  \sigma\stirlingsec{d-2l}{j+1}
-
2\sum_{l=0}^{\infty} \sigma \stirling{n}{k-2l}  \sigma\stirlingsec{k-2l}{j+1}.
\end{equation}
Here, the notation for the internal and external angle sums of the regular simplex introduced in Section~\ref{2320a} has been used.
\end{theorem}
In the special case when $k=d-1$, the above theorem combined with~\eqref{739} yields the following formula for the expected internal solid-angle sums at the $j$-dimensional faces of $\cP_{n,d}$.
\begin{corollary}\label{cor:angle_sum_P}
Fix some $d\in\N$ and $n\geq d+1$. For every $j\in\{0,\ldots,d-1\}$ the expected sum of internal angles of $\cP_{n,d}$ at its $j$-dimensional faces is given by
\begin{equation*}
\E \sum_{F\in\cF_j(\cP_{n,d})} \alpha(T_F(\cP_{n,d}))
=
\sum_{s=0}^{\infty} (-1)^s \sigma\stirling{n}{d-s}  \sigma\stirlingsec{d-s}{j+1}.
\end{equation*}
\end{corollary}

\vspace*{2mm}
Next we are going to state analogous results for convex hulls of random walks.
\begin{theorem}\label{1138thm}
Fix some $d\in\N$ and $n\geq d$. Then, for every
$j\in \{0,\ldots,d-1\}$ and $k\in \{0,\ldots,d\}$
the expected sum of the $k$-th Grassmann angles at the $j$-dimensional faces of $\cQ_{n,d}$ equals
\begin{equation}
\E \sum_{F\in\cF_j(\cQ_{n,d})}\gamma_k(T_F(\cQ_{n,d}))
= \frac{2\cdot j!}{n!} \sum_{l=0}^{\infty}\stirling{n+1}{d-2l}  \stirlingsec{d-2l}{j+1}
- \frac{2\cdot j!}{n!} \sum_{l=0}^{\infty}\stirling{n+1}{k-2l}  \stirlingsec{k-2l}{j+1}.
\end{equation}
Here, the notation for the Stirling numbers introduced in Section~\ref{2320b} has been used.
\end{theorem}

Let us mention some special and low-dimensional cases of Theorem~\ref{1138thm}.  Taking $k=d-1$ in Theorem~\ref{1138thm} and making use of~\eqref{739}, we compute the expected internal solid-angle sums at the $j$-dimensional faces of $\cQ_{n,d}$.
\begin{corollary}\label{cor:angle_sum_Q}
Fix some $d\in\N$ and $n\geq d$. Then, for every $j\in\{0,\ldots,d-1\}$ the expected sum of internal angles of $\cQ_{n,d}$ at its $j$-dimensional faces is given by
\begin{equation*}
\E \sum_{F\in\cF_j(\cQ_{n,d})} \alpha(T_F(\cQ_{n,d}))
= \frac{j!}{n!} \sum_{s=0}^{\infty} (-1)^s \stirling{n+1}{d-s}  \stirlingsec{d-s}{j+1}.
\end{equation*}
\end{corollary}
For  example, for $d=2$, the expected sum of angles of the random polygon  $\cQ_{n,2}$ at its vertices is given by
$$
\E \sum_{F\in\cF_0(\cQ_{n,2})} \alpha(T_F(\cQ_{n,2})) = \frac 1 {n!}\left(\stirling{n+1}{2}\stirlingsec{2}{1} - \stirling{n+1}{1}\stirlingsec{1}{1} \right)
=
H_n-1,
$$
where
$$
H_n = 1 + \frac 12 + \frac 13 +\ldots + \frac 1n
$$
is the $n$-th harmonic number. Since the angle sum of a polygon with $v$ vertices equals $(v-2)/2$ times the full solid angle $2\pi$, this agrees with the result of Baxter~\cite{baxter}, see also~\cite{baxter_nielsen} and~\cite[Lemma 4.1]{snyder_steele} for generalizations,  who proved that the expected number of vertices of $\cQ_{n,2}$ is
$$
\E f_{0}(\cQ_{n,2}) = 2 H_n.
$$

In dimension  $d=3$, the expected sum of internal angles of $\cQ_{n,3}$ at its vertices and a similar sum for edges are given by
\begin{align*}
\E \sum_{F\in\cF_0(\cQ_{n,3})} \alpha(T_F(\cQ_{n,3}))
&=
\frac 1 {n!}\left(\stirling{n+1}{3}\stirlingsec{3}{1} - \stirling{n+1}{2}\stirlingsec{2}{1} + \stirling{n+1}{1}\stirlingsec{1}{1} \right)\\
&=
\frac 12 (H_n)^2 - H_n -\frac 12 H_n^{(2)} + 1,\\
\E \sum_{F\in\cF_1(\cQ_{n,3})} \alpha(T_F(\cQ_{n,3}))
&=
\frac 1 {n!}\left(\stirling{n+1}{3}\stirlingsec{3}{2} - \stirling{n+1}{2}\stirlingsec{2}{2}  \right)\\
&=
\frac 32 (H_n)^2 - H_n -\frac 32 H_n^{(2)},
\end{align*}
where
$$
H_n^{(2)} = 1 + \frac 1{2^2} + \frac 1{3^2} +\ldots + \frac 1{n^2}.
$$
\begin{remark}
Using relations stated in Lemma~\ref{lem:identity} and in Remark~\ref{rem:stirling_relations}, below, one can rewrite Theorems~\ref{2219} and~\ref{1138thm} as follows:
\begin{align}
\E \sum_{F\in\cF_j(\cP_{n,d})}\gamma_k(T_F(\cP_{n,d}))
&=
2\sum_{l=1}^{\infty} \sigma \stirling{n}{k+2l}  \sigma\stirlingsec{k+2l}{j+1}
-
2\sum_{l=1}^{\infty}  \sigma\stirling{n}{d+2l}  \sigma\stirlingsec{d+2l}{j+1},\label{eq:alternative_P}\\
\E \sum_{F\in\cF_j(\cQ_{n,d})}\gamma_k(T_F(\cQ_{n,d}))
&=
 \frac{2\cdot j!}{n!}  \sum_{l=1}^{\infty}\stirling{n+1}{k+2l}  \stirlingsec{k+2l}{j+1}
-\frac{2\cdot j!}{n!} \sum_{l=1}^{\infty}\stirling{n+1}{d+2l}  \stirlingsec{d+2l}{j+1}. \label{eq:alternative_Q}
\end{align}
\end{remark}

\begin{remark}
Let us also mention one more result on angle sums of random polytopes. For typical cells in stationary tessellations, it is possible to compute  the expected angle-sums explicitly in terms of the cell intensities; see Theorem~10.1.3 and Equation~(10.4) in~\cite{SW08}.
\end{remark}

\subsection{Method of proof of Theorems~\ref{2219} and~\ref{1138thm}}
The main ingredient in the proofs of Theorems~\ref{2219} and~\ref{1138thm} is  the following stochastic representation of the Grassmann  angles of a polyhedral set. We recall that $W_k$ denotes a random, uniformly distributed linear random subspace of dimension $k$ in $\R^d$ and that $\Pi_k$ denotes the orthogonal projection on $W_k$.
The next theorem was stated by  Gr\"unbaum~\cite[p.~298]{bG68} with the comment that it is a simple application of the separation theorem for convex sets. Since its proof does not seem  trivial to us and since the result has been used many times since then (most notably, by Affentranger and Schneider~\cite{AS92}, see also~\cite[Section~8.3]{SW08}), we shall provide a proof in Sections~\ref{631} and~\ref{2157}.

\begin{theorem}\label{820}
Let $P\subset \R^d$ be a polyhedral set with non-empty interior. Then, for all integer $0\leq j < k \leq d$ and all  $F\in\cF_j(P)$ we have
\begin{align}\label{eq:gamma_k_proof}
\gamma_k(T_F(P)) = \P[\Pi_kF \not \in \cF(\Pi_k P)] = \P[\Pi_kF \not \in \cF_j(\Pi_k P)].
\end{align}
\end{theorem}
Taking the sum over all faces $F\in \cF_j(P)$ and noting that for almost every choice of $W_k$ every $j$-face of $\Pi_kP$ is the projection of some unique $j$-face of $P$ (which will be shown in Proposition~\ref{1640}) one arrives at the following
\begin{theorem}\label{623}
Let $P\subset \R^d$ be a polyhedral set with non-empty interior.  Then for all integer $0\leq j < k \leq d$  we have
\begin{align*}
\sum_{F\in\cF_j(P)}\gamma_k(T_F(P)) =f_j(P)-\E f_j(\Pi_k P).
\end{align*}
\end{theorem}

\vskip 10pt
The proofs of Theorems~\ref{820} and~\ref{623}  are postponed to Section~\ref{2157}. In Section~\ref{631} we will collect some properties of convex cones which are essential for these proofs. At this point, we provide the proofs of Theorems~\ref{2219} and~\ref{1138thm} assuming Theorem~\ref{623}.
\begin{proof}[Proof of Theorem~\ref{2219} assuming Theorem~\ref{623}]
First of all, let us establish the statement for all $j\in \{0,\ldots,d-1\}$ and $k\in \{0,\ldots,d\}$ such that $k\leq j$.
Since the lineality space of $T_F(\cP_{n,d})$ has dimension $j$ for every $F\in \cF_j(\cP_{n,d})$, which implies that $\gamma_k(T_F(\cP_{n,d})) = 1$ by~\eqref{eq:grassmann_angles_lineality}, we have
$$
\E \sum_{F\in\cF_j(\cP_{n,d})}\gamma_k(T_F(\cP_{n,d}))
=
\E f_j(\cP_{n,d})
=
2\sum_{l=0}^{\infty}  \sigma\stirling{n}{d-2l}  \sigma\stirlingsec{d-2l}{j+1},
$$
where in the last step we used~\eqref{eq:E_f_k_P_n_d}. This proves~\eqref{eq:theo_sum_grassmann_P} because the second term on the right-hand side there vanishes.

In the following, let $0\leq j < k \leq d$. Projecting $X_1,\ldots,X_n$ onto the random uniform $k$-plane $W_k$ gives $n$ independent standard Gaussian vectors in $W_k$ which can be identified with $\R^k$. Applying~\eqref{eq:E_f_k_P_n_d} to $\Pi_k\cP_{n,d}$ (which is the convex hull of $\Pi_kX_1,\ldots, \Pi_k X_n$) leads to
\begin{align*}
\E f_j(\Pi_k\cP_{n,d}) =
2\sum_{l=0}^{\infty} \sigma \stirling{n}{k-2l}  \sigma\stirlingsec{k-2l}{j+1}.
\end{align*}
On the other hand, for the original Gaussian polytope $\cP_{n,d}$ \eqref{eq:E_f_k_P_n_d}  states that
$$
\E f_j(\cP_{n,d})
=
2\sum_{l=0}^{\infty}  \sigma\stirling{n}{d-2l} \sigma\stirlingsec{d-2l}{j+1}.
$$
Combining  these two equations with  Theorem~\ref{623} completes the proof.
\end{proof}
\begin{remark}
An alternative way to prove Theorem~\ref{2219} is to apply Corollary~3.6 in~\cite{goetze_kabluchko_zaporozhets} to the tangent cones of the polytope $\cP_{n,d}$ which can be viewed as a Gaussian projection of the regular simplex. The Grassmann angles of the regular simplex appearing in that corollary can be computed using~\eqref{eq:crofton_conic} and~\eqref{eq:eq:upsilon_reg_simplex}, below. Note also that the case when $n\leq d$ omitted in Theorem~\ref{2219} (meaning that $\cP_{n,d}$ is a simplex of dimension $n-1$ in $\R^d$), was treated in Theorem~4.1 of~\cite{goetze_kabluchko_zaporozhets}. Translated into the notation of the present paper, this result shows that~\eqref{eq:alternative_P} (but not Theorem~\ref{2219}) continues to hold under the assumptions $d\in\N$, $n\in \{2,\ldots, d\}$, $j,k\in \{0,\ldots,n-2\}$. Let us also mention that Theorem~\ref{2219} is related to Theorems~1.12 and~1.13 of~\cite{beta_polytopes}, where the expected conic intrinsic volumes of the tangent cones of the so-called beta polytopes have been computed. The Gaussian polytopes considered here can be viewed as the limiting case $\beta\to+\infty$ of the beta polytopes.
\end{remark}

\begin{remark}
All results on the polytope $\cP_{n,d}$ remain true if it is replaced by the random polytope $\cP_{n,d}'$ defined as a random projection of the regular simplex $\conv(e_1,\ldots,e_n)$ onto a random uniform $d$-dimensional  subspace in $\R^n$. Indeed, \eqref{eq:E_f_k_P_n_d} remains true for $\cP'_{n,d}$ by the original result of~\cite{AS92}, and a projection of $\cP_{n,d}'$ onto a random uniform subspace of dimension  $k < d$ has the same distribution as $\cP_{n,k}'$, so that the above proof applies.
\end{remark}

\begin{proof}[Proof of Theorem~\ref{1138thm} assuming Theorem~\ref{623}]
In the case $k\leq j$ the statement can be proven in the same way as in the proof of Theorem~\ref{2219}, but this time we have to appeal to~\eqref{800}.  In the following, let  $0\leq j < k \leq d$.
Projecting the path $S_0,\ldots,S_n$ onto the random $k$-plane $W_k$ gives a random walk in $W_k$. We can identify $W_k$ with $\R^k$. The increments of the projected random walk are given by
\begin{align*}
\xi'_1 := \Pi_k \xi_{1}, \;\;\; \ldots,\;\;\; \xi'_{n} := \Pi_k \xi_n.
\end{align*}
It is straightforward to check that the projected random walk satisfies conditions $(\text{Ex})$ and $(\text{GP})$ as well. In particular, for $(\text{GP})$ note that any $k$ vectors among $S_1,\ldots,S_n$ are a.s.\ linearly independent (since $k\leq d$ and $(\text{GP})$ holds for the original random walk), hence their projections onto an independent $k$-plane $W_k$ are also linearly independent a.s.  Therefore applying~\eqref{800} leads to
\begin{align*}
\E f_j(\Pi_k\cQ_{n,d}) = \frac{2\cdot j!}{n!} \sum_{l=0}^{\infty}\stirling{n+1}{k-2l}  \stirlingsec{k-2l}{j+1}.
\end{align*}
On the other hand, \eqref{800} applied to the original random walk states that
$$
\E f_j(\cQ_{n,d})= \frac{2\cdot j!}{n!} \sum_{l=0}^{\infty}\stirling{n+1}{d-2l}  \stirlingsec{d-2l}{j+1}.
$$
Combining these two equations  with Theorem~\ref{623} completes the proof.
\end{proof}

\subsection{Expected sums of  conic intrinsic volumes}
From the above Theorems~\ref{2219} and~\ref{1138thm} we can deduce formulas for the expected sums of conic intrinsic volumes of the tangent cones of the random polytopes $\cP_{n,d}$ and $\cQ_{n,d}$. Given a polyhedral cone $C\subset \R^d$, its $k$-th \emph{conic intrinsic volume} $\upsilon_k(C)$ is defined as
\begin{equation}\label{eq:upsilon_def}
\upsilon_k(C)	
=
\sum_{F\in\cF_k(C)}\alpha(F)\alpha(N_F(C)),
\end{equation}
for $k\in \{0,\ldots,d\}$.
There are other equivalent definitions using, for example,  the conic Steiner formula or Euclidean projections;  see~\cite[Section~6.5]{SW08}, \cite{glasauer_phd}, \cite{amelunxen_comb}, \cite{amelunxen_edge}, \cite[Section 2]{HugSchneider2016}.

\begin{theorem}\label{2219v}
Fix some $d\in\N$ and $n\geq d+1$. Then, for all $j\in\{0,\ldots,d-1\}$ and  $k\in \{j,\ldots,d-1\}$  we have
\begin{align*}
\E \sum_{F\in\cF_j(\cP_{n,d})}\upsilon_k(T_F(\cP_{n,d})) = \sigma \stirling{n}{k+1} \sigma\stirlingsec{k+1}{j+1}.
\end{align*}
In the remaining case when $j\in\{0,\ldots,d-1\}$ and  $k = d$  we have
\begin{align*}
\E \sum_{F\in\cF_j(\cP_{n,d})}\upsilon_d(T_F(\cP_{n,d}))
&=
\sum_{s=0}^\infty (-1)^{s} \sigma\stirling{n}{d-s}  \sigma \stirlingsec{d-s}{j+1}
=
\sum_{s=1}^\infty (-1)^{s+1} \sigma\stirling{n}{d+s}  \sigma \stirlingsec{d+s}{j+1}
.
\end{align*}
\end{theorem}

\begin{theorem}\label{1138v}
Fix some $d\in\N$ and $n\geq d$. Then, for all $j\in\{0,\ldots,d-1\}$ and  $k\in \{j,\ldots,d-1\}$  we have
\begin{equation*}
\E \sum_{F\in\cF_j(\cQ_{n,d})}\upsilon_k(T_F(\cQ_{n,d}))= \frac{j!}{n!}\stirling{n+1}{k+1}\stirlingsec{k+1}{j+1}.
\end{equation*}
In the remaining case when $j\in\{0,\ldots,d-1\}$ and  $k = d$  we have
\begin{align*}
\E \sum_{F\in\cF_j(\cQ_{n,d})}\upsilon_d(T_F(\cQ_{n,d}))
&=
\frac{j!}{n!}\sum_{s=0}^{\infty} (-1)^s\stirling{n+1}{d-s}\stirlingsec{d-s}{j+1}
=
\frac{j!}{n!} \sum_{s=1}^{\infty} (-1)^{s+1} \stirling{n+1}{d+s}\stirlingsec{d+s}{j+1}
.
\end{align*}
\end{theorem}
Note that in both theorems, the case $k=d$ yields a formula for the expected sum of internal angles of $\cP_{n,d}$ and $\cQ_{n,d}$ already (partially) obtained in Corollaries~\ref{cor:angle_sum_P} and~\ref{cor:angle_sum_Q}. On the other extreme, taking $k=j$ and noting that $\upsilon_{j}(T_F(P)) = \alpha(N_F(P))$ for all $F\in \cF_j(P)$ because the only face of dimension $j$ in $T_F(P)$ is its lineality space (which is a shift of $\aff F$),  we obtain the following expressions for the sums of the external angles.

\begin{corollary}\label{cor_ext_angle_sum_P}
Fix some $d\in\N$ and $n\geq d+1$. Then, for every $j\in\{0,\ldots,d-1\}$  we have
\begin{align*}
\E \sum_{F\in\cF_j(\cP_{n,d})} \alpha (N_F(\cP_{n,d})) = \sigma\stirling{n}{j+1}.
\end{align*}
\end{corollary}
\begin{corollary}\label{cor_ext_angle_sum_Q}
Fix some $d\in\N$ and $n\geq d$. Then, for every $j\in\{0,\ldots,d-1\}$  we have
\begin{equation*}
\E \sum_{F\in\cF_j(\cQ_{n,d})} \alpha(N_F(\cQ_{n,d}))= \frac{j!}{n!}\stirling{n+1}{j+1}.
\end{equation*}
\end{corollary}

For the proof of Theorems~\ref{2219v} and~\ref{1138v} we  use a relation, known as the \emph{conic Crofton formula},  between the Grassmann angles of a cone and its conical intrinsic volumes. Precisely, according to~\cite[p.~261]{SW08} we have
\begin{align}\label{eq:crofton_conic}
\gamma_k(C)=2\sum_{i=1,3,5,\ldots} \upsilon_{k+i}(C)
\end{align}
for every cone $C\subset \R^d$ which is not a linear subspace and for all $k\in \{0,\ldots,d\}$. Consequently,
\begin{align}\label{relation_quer_intr}
\upsilon_d(C)=\frac{1}{2}\gamma_{d-1}(C),
\;\;\;
\upsilon_k(C) = \frac{1}{2}\gamma_{k-1}(C)-\frac{1}{2}\gamma_{k+1}(C),
\end{align}
for all  $k\in\{0,\ldots,d-1\}$. Here, in the case $k=0$ we have to define $\gamma_{-1}(C) = 1$ and the proof of~\eqref{relation_quer_intr} follows from~\eqref{eq:crofton_conic} together with the identity $\upsilon_0(C) + \upsilon_2(C) +\ldots = 1/2$.

\begin{proof}[Proof of Theorem~\ref{2219v}]
Let us start by observing that in the case $k=0$ (which implies $j=0$), we can use the fact that the external angles at the vertices of any polytope sum up to $1$. This yields
$$
\E \sum_{F\in\cF_0(\cP_{n,d})}\upsilon_0(T_F(\cP_{n,d}))
=
1
=
\sigma\stirling{n}{1}\sigma\stirlingsec{1}{1},
$$
which is the desired result. In the following we exclude the case $k=j=0$.

In the general case, we can use the linear relation~\eqref{relation_quer_intr} between the Grassmann angles $\gamma_k$ and the conic intrinsic volumes $\upsilon_k$. Then, applying Theorem~\ref{2219}, it follows that for all $k\in\{j,\ldots,d-1\}$,
\begin{align*}
\E \sum_{F\in\cF_j(\cP_{n,d})}\upsilon_k(T_F(\cP_{n,d}))
&=\E \sum_{F\in\cF_j(\cP_{n,d})}\Big(\frac{1}{2}\gamma_{k-1}(T_F(\cP_{n,d}))- \frac{1}{2}\gamma_{k+1}(T_F(\cP_{n,d}))\Big)\\
&= \sum_{l=0}^{\infty} \sigma \stirling{n}{k-2l+1} \sigma \stirlingsec{k-2l+1}{j+1}- \sum_{l=0}^{\infty} \sigma \stirling{n}{k-2l-1}\sigma\stirlingsec{k-2l-1}{j+1}\\
&= \sigma \stirling{n}{k+1} \sigma \stirlingsec{k+1}{j+1}.
\end{align*}
In the case $k=d$, we get
\begin{align*}
\E \sum_{F\in\cF_j(\cP_{n,d})}\upsilon_{d}(T_F(\cP_{n,d}))&=\E \sum_{F\in\cF_j(\cP_{n,d})}\frac{1}{2}\gamma_{d-1}(T_F(\cP_{n,d}))\\
&=\sum_{l=0}^\infty       \sigma\stirling{n}{d-2l}\sigma\stirlingsec{d-2l}{j+1} - \sum_{l=0}^\infty \sigma\stirling{n}{d-1-2l}\sigma\stirlingsec{d-1-2l}{j+1}\\
&=\sum_{s=0}^\infty(-1)^s \sigma\stirling{n}{d-s} \sigma\stirlingsec{d-s}{j+1}.
\end{align*}
The second formula in the case $k=d$ follows then from the identity (see Lemma~\ref{lem:identity}, below)
$$
\sum_{m=j+1}^{n} (-1)^{n-m} \sigma \stirling{n}{m} \sigma\stirlingsec{m}{j+1} = \delta_{n,j+1}
$$
together with the observation that the Kronecker symbol on the right-hand side vanishes because $j+1 \leq d < n$.
\end{proof}

\begin{lemma}\label{lem:identity}
For all $n,k\in\N$ with $n\geq k$ we have
$$
\sum_{m=k}^{n} (-1)^{n-m} \sigma \stirling{n}{m} \sigma\stirlingsec{m}{k} = \delta_{n,k},
\;\;\;
\sum_{m=k}^{n} \sigma \stirling{n}{m} \sigma\stirlingsec{m}{k} = \binom nk.
$$
\end{lemma}
\begin{proof}
To prove the identity, consider the tangent cone of the regular simplex $\Delta_n= \conv(e_1,\ldots,e_n)$ at its face $\Delta_k=\conv(e_1,\ldots,e_k)$. Its $(m-1)$-st conic intrinsic volume can be computed using formula~\eqref{eq:upsilon_def} by observing that the $(m-1)$-dimensional faces of the tangent cone correspond to the $m$-vertex faces of $\Delta_n$ containing $\Delta_k$ and that the internal (respectively, normal) angles at these faces correspond to the internal (respectively, external) angles of these  faces. Since the number of such $m$-vertex faces in $\binom {n-k}{m-k}$,  one obtains the following formula (which can be found already in~\cite{AS92}):
\begin{equation}\label{eq:eq:upsilon_reg_simplex}
\upsilon_{m-1}(T_{\Delta_k}(\Delta_n))
=
\binom {n-k}{m-k} \frac{\sigma\stirlingsec{m}{k}}{\binom mk} \frac{\sigma\stirling{n}{m}}{\binom nm}
=
\frac 1 {\binom nk}\sigma \stirling{n}{m} \sigma \stirlingsec{m}{k},
\end{equation}
for all $m\in \{k,\ldots,n\}$. Moreover, the intrinsic volumes $\upsilon_{m-1}(T_{\Delta_k}(\Delta_n))$ with $m\in \{0,\ldots, k-1\}$ vanish because all faces of $T_{\Delta_k}(\Delta_n)$ have dimension at least $k-1$. The claim of the lemma follows from the identities $\sum_{m=1}^n \upsilon_{m-1}(C) = 1$ and $\sum_{m=1}^{n}  (-1)^m \upsilon_{m-1}(C) = 0$ that are valid for every $(n-1)$-dimensional polyhedral cone $C$ which is not a linear subspace. Let us also note that Lemma~\ref{lem:identity} can be viewed as the limiting case, as $\beta\to +\infty$, of the identities for the expected angle sums of the random beta simplices stated in~\cite[Proposition~2.1]{kabluchko_algorithm}.
\end{proof}

\begin{remark}\label{rem:stirling_relations}
Identities similar to those stated in Lemma~\ref{lem:identity} are well known for Stirling numbers. Namely, for all $n,k\in\N$ with $n\geq k$, we have
$$
\sum_{m=k}^{n} (-1)^{n-m} \stirling{n}{m} \stirlingsec{m}{k} = \delta_{n,k},
\;\;\;
\sum_{m=k}^{n}  \stirling{n}{m} \stirlingsec{m}{k} = L(n,k),
$$
where $L(n,k) = \frac {n!}{k!} \binom {n-1}{k-1}$ are the Lah numbers. An even more interesting analogy between the Stirling numbers and the angles of the regular simplex is related to the identity $\stirlingsec{n}{k} = \stirling{-k}{-n}$ which becomes valid after a natural extension of the Stirling numbers to negative parameters~\cite[\S6.1]{Graham1994}.  It follows directly from~\eqref{eq:regular_simpl_external} and~\eqref{eq:regular_simpl_internal} that the individual angles of the regular simplex (rather than the angle sums $\sigma \stirling{n}{k}$ and $\sigma\stirlingsec{n}{k}$) satisfy a similar identity. Given these analogies, one may ask whether the Stirling numbers can be interpreted as angles of some polytope. This is indeed the case and it turns out that this polytope is the Schl\"afli orthoscheme. These questions will be studied in more detail elsewhere.
\end{remark}

\begin{proof}[Proof of Theorem~\ref{1138v}]
Let us first assume that $k\neq 0$.
Again, we can use the linear relation~\eqref{relation_quer_intr} and obtain for $k\in\{j,\ldots,d-1\}$,
\begin{align*}
\E \sum_{F\in\cF_j(\cQ_{n,d})}\upsilon_k(T_F(\cQ_{n,d}))
&	=\E \sum_{F\in\cF_j(\cQ_{n,d})}\Big(\frac{1}{2}\gamma_{k-1}(T_F(\cQ_{n,d}))-\frac{1}{2} \gamma_{k+1}(T_F(\cQ_{n,d}))\Big)\\
&	=\frac{j!}{n!}\sum_{l=0}^{\infty}\stirling{n+1}{k-2l+1}\stirlingsec{k-2l+1}{j+1}-\frac{j!}{n!}\sum_{l=0}^{\infty}\stirling{n+1}{k-2l-1}\stirlingsec{k-2l-1}{j+1}\\
&	=\frac{j!}{n!}\stirling{n+1}{k+1}\stirlingsec{k+1}{j+1},
\end{align*}
where we applied Theorem~\ref{1138thm} twice. 
For $k=d$, relation~\eqref{relation_quer_intr} and Theorem~\ref{1138thm} yield
\begin{align*}
\E \sum_{F\in\cF_j(\cQ_{n,d})}\upsilon_{d}(T_F(\cQ_{n,d}))
&=\E \sum_{F\in\cF_j(\cQ_{n,d})}\frac{1}{2}\gamma_{d-1}(T_F(\cQ_{n,d}))\\
&=\frac{j!}{n!}\sum_{l=0}^\infty\stirling{n+1}{d-2l}\stirlingsec{d-2l}{j+1}-\frac{j!}{n!}\sum_{l=0}^\infty\stirling{n+1}{d-1-2l}		\stirlingsec{d-1-2l}{j+1}\\
&=\frac{j!}{n!}\sum_{s=0}^\infty(-1)^s\stirling{n+1}{d-s}\stirlingsec{d-s}{j+1}.
\end{align*}
The second formula in the case $k=d$ follows then from the identity
$$
\sum_{m=j+1}^{n+1} (-1)^{n+1-m} \stirling{n+1}{m}\stirlingsec{m}{j+1} = \delta_{n,j},
$$
where the Kronecker symbol on the right hand-side vanishes  because $j\leq d-1 <n$.

In order to treat the remaining case $k=0$ (which implies that $j=0$), we make use of the fact that the sum of external angles at all vertices in any polytope is $1$, hence
$$
\E \sum_{F\in\cF_0(\cQ_{n,d})}\upsilon_0(T_F(\cQ_{n,d}))
=
1
=
\frac{0!}{n!}\stirling{n+1}{1}\stirlingsec{1}{1},
$$
which is the desired result.
\end{proof}

\subsection{Expected angle sums of Gaussian projections of polyhedral sets} \label{subsec:gauss_proj}

In this section, we are going to see that for polyhedral set $P\subset \R^n$ and  a Gaussian random matrix $A\in\R^{d\times n}$ (meaning that the entries of $A$ are independent and standard Gaussian distributed random variables), the angle sums of the so-called Gaussian projection $AP$ can be expressed in terms of the angle sums of $P$. As we shall see below, this setting includes the angle sums of Gaussian polytopes and some other interesting examples as special cases.

\begin{theorem} \label{theorem:grassmann_sums_proj_polytope}
Fix some $d\in\N$ and $n\ge d$. Let $P\subset \R^n$ be a polyhedral set with non-empty interior and let $A\in\R^{d\times n}$ be a Gaussian matrix. Then, we have
\begin{align*}
\E\sum_{F\in\cF_j(AP)}\gamma_k(T_F(AP))=\sum_{G\in\cF_j(P)}\big(\gamma_k(T_G(P))-\gamma_d(T_G(P))\big)
\end{align*}
for all $j\in\{0,\dots,d-1\}$ and $k\in\{j-1,\dots,d-1\}$ and we recall that $\gamma_{-1} (C) = 1$.
\end{theorem}

The proof of Theorem~\ref{theorem:grassmann_sums_proj_polytope} is postponed to Section~\ref{2157}.

\begin{corollary}\label{cor:angle_sum_gauss_proj}
Fix some $d\in\N$ and $n\ge d$. Let $P\subset \R^n$ be a polyhedral set with non-empty interior and let $A\in\R^{d\times n}$ be a Gaussian matrix. Then, for all $j\in\{0,\dots,d-1\}$ and $k\in\{j,\dots,d-1\}$ we have
\begin{align*}
\E\sum_{F\in\cF_j(AP)}\upsilon_k(T_F(AP))=\sum_{G\in\cF_j(P)}\upsilon_k(T_G(P)).
\end{align*}
In the remaining case when $j\in\{0,\dots,d-1\}$ and $k=d$, we obtain
\begin{align*}
\E\sum_{F\in\cF_j(AP)}\upsilon_d(T_F(AP))
	&=\sum_{G\in\cF_j(P)}\sum_{s=0}^{n-d}(-1)^s\upsilon_{d+s}(T_G(P)).
\end{align*}
\end{corollary}

\begin{proof}

We use Theorem~\ref{theorem:grassmann_sums_proj_polytope} and the linear relation~\eqref{relation_quer_intr} between the Grassmann angles $\gamma_k$ and the conic intrinsic volumes $\upsilon_k$. Thus, we obtain
\begin{align*}
\E\sum_{F\in\cF_j(AP)}\upsilon_k(T_F(AP))
&	=\frac{1}{2}\E\left[\sum_{F\in\cF_j(AP)}\big(\gamma_{k-1}(T_F(AP))-\gamma_{k+1}(T_F(AP))\big)\right]\\
&	=\frac{1}{2}\sum_{G\in\cF_j(P)}\big(\gamma_{k-1}(T_G(P))-\gamma_d(T_G(P))-\gamma_{k+1}(T_G(P))+\gamma_d(T_G(P))\big)\\
&	=\frac{1}{2}\sum_{G\in\cF_j(P)}\big(\gamma_{k-1}(T_G(P))-\gamma_{k+1}(T_G(P))\big)\\
&	=\sum_{G\in\cF_j(P)}\upsilon_k(T_G(P))
\end{align*}
for $k\in\{j,\dots,d-1\}$.
Note that we used that both, $T_F(AP)$ and $T_G(P)$, are not linear subspaces. To justify this, note first that  both cones are full-dimensional since $\dim P = n$ and $\dim AP = d$ (because the rank of $A$ is $d$ which we shall show in the proof of Theorem~\ref{theorem:grassmann_sums_proj_polytope}). To complete the argument, note that the lineality spaces of  $T_F(AP)$ and $T_G(P)$ have dimension $j$, which is strictly smaller than $n$ and $d$.


In the remaining case $k=d$, we use~\eqref{relation_quer_intr} combined with Theorem~\ref{theorem:grassmann_sums_proj_polytope} again and obtain
\begin{align*}
\E\sum_{F\in\cF_j(AP)}\upsilon_d(T_F(AP))
&	= \frac 12\, \E\sum_{F\in\cF_j(AP)}\gamma_{d-1}(T_F(AP))\\
&	=\sum_{G\in\cF_j(P)}\frac{1}{2}\big(\gamma_{d-1}(T_G(P))-\gamma_{d}(T_G(P))\big).
\end{align*}
Applying~\eqref{eq:crofton_conic} yields
\begin{align*}
\E\sum_{F\in\cF_j(AP)}\upsilon_d(T_F(AP))
&	=\sum_{G\in\cF_j(P)}\sum_{i=1,3,5,\dots}\big(\upsilon_{d-1+i}(T_G(P))-\upsilon_{d+i}(T_G(P))\big)\\
&	=\sum_{G\in\cF_j(P)}\sum_{s=0}^\infty(-1)^s\upsilon_{d+s}(T_G(P)),
\end{align*}
which completes the proof.
\end{proof}

Let us consider some special cases of Corollary~\ref{cor:angle_sum_gauss_proj}.
\begin{remark}
In the case $P=\conv(0,e_1,e_1+e_2,\dots,e_1+\ldots+e_n)$, where $e_1,\dots,e_n$ denotes the standard Euclidean basis vectors in $\R^n$, we observe that $AP$, for a Gaussian matrix $A\in\R^{d\times n}$, has the same distribution as the convex hull of a random walk in $\R^d$ with independent standard Gaussian increments. Using Corollary~\ref{cor:angle_sum_gauss_proj} and the formula of Theorem~\ref{1138v}, we obtain
\begin{align*}
\sum_{G\in\cF_j(P)}\upsilon_k(T_G(P))
=\E\sum_{F\in\cF_j(AP)}\upsilon_k(T_F(AP))
=\frac{j!}{n!}\stirling{n+1}{k+1}\stirlingsec{k+1}{j+1}
\end{align*}
for $0\le j\le k\le d-1$. Note that $P$ is also called the Schl\"afli orthoscheme of type $B$, which was extensively studied in~\cite{GK2020_Schlaefli_orthoschemes}. The above formula recovers Theorem 3.1 from~\cite{GK2020_Schlaefli_orthoschemes}.
\end{remark}

\begin{remark}
In the case of the regular simplex $P=\conv(e_1,e_2,\dots,e_n)$ with  $n\geq d+1$, we observe that $AP$, for a Gaussian matrix $A\in\R^{d\times n}$, has the same distribution as the Gaussian polytope $\cP_{n,d}$.
Using Corollary~\ref{cor:angle_sum_gauss_proj} and the formula of Theorem~\ref{2219v}, we obtain
\begin{align*}
\sum_{G\in\cF_j(P)}\upsilon_k(T_G(P))
=\E\sum_{F\in\cF_j(\cP_{n,d})}\upsilon_k(T_F(\cP_{n,d}))=\sigma\stirling{n}{k+1}\sigma\stirlingsec{k+1}{j+1}
\end{align*}
for $0\le j\le k\le d-1$. Since the simplex $P$ is not full-dimensional, we have to apply Corollary~\ref{cor:angle_sum_gauss_proj} with $\R^n$ replaced by the affine hull of $P$ which has dimension $n-1$.  Of course, the same argument could be used in the other direction, in which case we would recover Theorem~\ref{2219v}.
\end{remark}

\begin{remark}
For independent and standard Gaussian distributed random vectors $\xi_1,\dots,\xi_n$ with values in $\R^d$, where $n\geq d$, we define $D=\pos(\xi_1,\dots,\xi_n)$. Then, the random cone $D$ has the same distribution as $A\R^n_+$ for a Gaussian matrix $A\in\R^{d\times n}$. Thus, we obtain
\begin{align*}
\E\sum_{G\in\cF_j(D)}\upsilon_k(T_G(D))
=\sum_{F\in\cF_j(\R^n_+)}\upsilon_k(T_F(\R^n_+))=\binom{n}{j}\binom{n-j}{k-j}2^{j-n}
\end{align*}
for all $0\le j\le k\le d-1$. In order to prove the last step, we need to consider the tangent cones $T_F(\R^n_+)$ for $F\in\cF_j(\R^n_+)$. Each $j$-face $F\in\cF_j(\R^n_+)$ is determined by a collection of indices $1\le i_1<\ldots<i_{n-j}\le n$ and given by
\begin{align*}
F=\{(x_1,\dots,x_n)\in\R^n:x_{i_1}=\ldots=x_{i_{n-j}}=0,\;x_l\ge 0\; \text{ for all } l\notin\{i_1,\dots,i_{n-j}\}\}.
\end{align*}
Then, the corresponding tangent cone is given by
\begin{align*}
T_F(\R^n_+)=\{v\in\R^n:v_{i_1}\ge 0,\dots,v_{i_{n-j}}\ge 0\}
\end{align*}
which is isometric  to $\R^{n-j}_+\times\R^j$. Using the well-known formula for the conic intrinsic volumes of the orthant $\R^n_+$, see e.g.~\cite[Example~2.8]{amelunxen_comb}, we obtain
\begin{align*}
\sum_{F\in\cF_j(\R^n_+)}\upsilon_k(T_F(\R^n_+))=\sum_{F\in\cF_j(\R^n_+)}\upsilon_{k-j}(\R^{n-j}_+)
=\sum_{F\in\cF_j(\R^n_+)}\binom{n-j}{k-j}2^{j-n}=\binom{n}{j}\binom{n-j}{k-j}2^{j-n}.
\end{align*}
In the remaining case $k=d$, we get
$$
\E\sum_{G\in\cF_j(D)}\upsilon_d(T_G(D))
=\sum_{F\in\cF_j(\R^n_+)}\sum_{s=0}^{n-d}(-1)^s\upsilon_{d+s}(T_F(\R^n_+))
=
2^{j-n} \binom nj \sum_{s=0}^{n-d}(-1)^s
\binom{n-j}{d+s-j}.
$$
\end{remark}

\subsection{Invariance of angle sums under affine transformations}
In general, the solid-angle sums  of a deterministic polytope (as well as the more general sums of Grassmann angles), are not invariant under affine transformations of the ambient space. For example, for a simplex in dimension at least $3$, the sum of solid angles at vertices can take any value between $0$ and $1/2$ (see~\cite{PS67}), although all simplices can be transformed to each other by affine transformations.  The next proposition states that if the polytope is random and its law is rotationally invariant, then the expected angle-sums become affine invariant.
\begin{theorem}\label{prop:elliptic}
Let $P$ be a random polytope (or, more generally, polyhedral set) with a.s.\ non-empty interior  in $\R^d$. Assume that the law of $P$ is  invariant under orthogonal maps, that is $OP$ has the same distribution as $P$ for every deterministic orthogonal transformation $O:\R^d\to\R^d$. Let $A:\R^d\to\R^d$ be a deterministic linear map with $\det A\neq 0$. Then, for all $j\in \{0,\ldots,d-1\}$ and $k\in \{0,\ldots,d\}$,
\begin{align}
&\E \sum_{G\in \cF_j(AP)} \gamma_k(T_{G}(AP)) = \E \sum_{F\in \cF_j(P)} \gamma_k(T_{F}(P)), \label{eq:prop:elliptic1}\\
&\E \sum_{G\in \cF_j(AP)} \upsilon_k(T_{G}(AP)) = \E \sum_{F\in \cF_j(P)} \upsilon_k(T_{F}(P)). \label{eq:prop:elliptic2}
\end{align}
In the special case $k=d$, \eqref{eq:prop:elliptic2} implies that  the expected angle-sums are invariant in the sense that for all $j\in \{0,\ldots,d-1\}$:
$$
\E \sum_{G\in \cF_j(AP)} \alpha(T_{G}(AP)) = \E \sum_{F\in \cF_j(P)} \alpha(T_{F}(P)).
$$
\end{theorem}
Since an arbitrary non-degenerate Gaussian distribution can be represented  as a linear image of the standard Gaussian distribution, the above theorem yields the following
\begin{corollary}
Theorems~\ref{2219}, \ref{2219v} and Corollaries~\ref{cor:angle_sum_P}, \ref{cor_ext_angle_sum_P} remain true if the points $X_1,\ldots,X_n$ generating the polytope  $\cP_{n,d}$ are sampled independently from an arbitrary non-degenerate Gaussian distribution.
\end{corollary}

\begin{proof}[Proof of Theorem~\ref{prop:elliptic}]
Let first $Q\subset \R^d$ be any deterministic polytope with non-empty interior and  $A:\R^d\to\R^d$ a linear map with $\det A\neq 0$. All $j$-dimensional faces of $AQ$ are of the form  $G= A F$ for some $F\in \cF_j(Q)$. The tangent cone of the polytope $AQ$ at its face $AF$ coincides with $A (T_F(Q))$. If $W_{d-k}$ denotes a random uniform $(d-k)$-plane in $\R^d$ which is independent of everything else, then the $k$-th Grassmann angle of $AQ$ at $AF$ can be written as
\begin{multline*}
\gamma_k (T_{AF}(AQ))
=
\gamma_k (A T_F(Q))
=
\P[W_{d-k} \cap A T_F(Q) \neq \{0\}]
=
\P[A^{-1} W_{d-k} \cap T_F(Q) \neq \{0\}]
\\
=
\int_{G(d,d-k)} \ind_{\{V\cap  T_F(Q) \neq \{0\}\}} \P_{A^{-1}W_{d-k}}(\dd V),
\end{multline*}
where $\P_{A^{-1}W_{d-k}}$ is the probability law of $A^{-1}W_{d-k}$ on $G(d,d-k)$, the Grassmannian of linear $(d-k)$-planes in $\R^d$.   Taking the sum over all $F\in \cF_j(Q)$, we obtain
$$
\sum_{G = AF \in \cF_j(AQ)} \gamma_k (T_{G}(AQ))
=
\int_{G(d,d-k)} \left(\sum_{F\in \cF_j(Q)} \ind_{\{V\cap  T_F(Q) \neq \{0\}\}}\right) \P_{A^{-1}W_{d-k}}(\dd V).
$$
Applying this to $Q=P$ (with $W_{d-k}$ being independent of $P$), taking the expectation, and using Fubini's theorem we get
$$
\E \sum_{G\in \cF_j(AP)} \gamma_k (T_{G}(AP))
=
\int_{G(d,d-k)} \E \left[\sum_{F\in \cF_j(P)} \ind_{\{V\cap  T_F(P) \neq \{0\}\}}\right] \P_{A^{-1}W_{d-k}}(\dd V).
$$
However, since the probability law of the random polytope $P$ is rotationally invariant, the expectation inside the integral does not depend on the choice of $V\in G(d,d-k)$ and it follows that
$$
\E \sum_{G\in \cF_j(AP)} \gamma_k (T_{G}(AP))
=
\E \left[\sum_{F\in \cF_j(P)} \ind_{\{V_0\cap  T_F(P) \neq \{0\}\}}\right],
$$
where $V_0$ is any element of $G(d,d-k)$. Observe that the right-hand side does not depend on the choice of the linear map $A$. Since we can apply the above argument to the case when $A$ is the identity map, we arrive at the identity
$$
\E \sum_{G\in \cF_j(AP)} \gamma_k (T_{G}(AP))  = \E \sum_{F\in \cF_j(P)} \gamma_k (T_{F}(P)),
$$
which proves~\eqref{eq:prop:elliptic1}. To prove~\eqref{eq:prop:elliptic2}, recall~\eqref{relation_quer_intr}.
\end{proof}

\section{Linear images of polyhedral sets}\label{631}
In this section, we prove some facts on linear images (including projections)  of polyhedral sets which will be used in the proof of Theorem~\ref{820}.
Consider a polyhedral set $P\subset \R^d$ with a non-empty interior. Take some $k\in \{1,\ldots,d\}$ and let $A:\R^d\to \R^k$ be a linear map of full rank $k$, which means that $\Ima A = \R^k$ or, equivalently, $\dim \Ker A = d-k$. We are interested in relating the faces of the polyhedral set $AP$ to the faces of the original polyhedral set $P$.
The first main result of the present section, Proposition~\ref{1640}, states that every proper face of  $AP$ is an image of some face of $P$.
However, the converse is not true: not every face of $P$ is mapped to a face of $AP$.
The second main result, Proposition~\ref{1122}, states several equivalent conditions which guarantee that the image of a face of $P$ is a face of $AP$. These results require a general position assumption on $\Ker A$ with respect to  $P$, which we are now going to state.

Let $M$ be a convex set in $\R^d$. Denote by $L$ the unique linear subspace in $\R^d$ such that for some $t\in\R^d$,
\begin{align*}
\aff M = t+ L.
\end{align*}
In other words, $L$ is the translation of the affine hull of $M$ passing through the origin.
We say that $M$ is  \textit{in general position with respect to a linear subspace} $L'\subset \R^d$ if
\begin{align*}
\dim (L \cap L') = \max(\dim L-\codim L',0).
\end{align*}
Also, we say that a linear subspace $L'\subset \R^d$ is \textit{in general position with respect to a polyhedral set} $P$ if it is in general position with respect to all faces of $P$ of all dimensions.

\begin{lemma}\label{1641}
Let $M$ be a convex  set in $\R^d$. Fix some $k\in\{1,\ldots,d\}$ and let  $A:\R^d\to\R^k$ be a linear map of full rank $k$ (that is, $\dim\Ker A=d-k$). If $\Ker A$ is in general position with respect to $M$, then
\begin{align*}
    \dim A M=\min(k,\dim M).
\end{align*}
\end{lemma}
\begin{proof}
Let $\aff M=t+L$, where $t\in\R^d$ and $L\subset\R^d$ is a linear subspace. Since the map $A$ preserves affine and linear hulls, we have $\dim AM = \dim AL$. To prove the proposition, we need to show that
\begin{align*}
    \dim A L=m, \quad\text{where}\quad m:=\min(k,\dim L).
\end{align*}
Since $\codim\Ker A = k$, it follows from the general position assumption that
\begin{align}\label{936}
    \dim (L\cap \Ker A)=\dim L-m.
\end{align}
This implies that  there exist linearly independent vectors $e_1,\ldots,e_m\in L$ such that
\begin{align*}
    \lin(e_1,\ldots,e_m)\cap \Ker A=\{0\}.
\end{align*}
Therefore, for any tuple $(c_1,\ldots,c_m)\ne (0,\ldots,0)$,
\begin{align*}
    0\ne A(c_1e_1+\ldots+c_me_m)=c_1Ae_1+\ldots+c_mAe_m,
\end{align*}
which implies the linear independence of $Ae_1,\ldots,Ae_m$. Thus, $\dim AL\geq m$.
On the other hand, we obviously have $\dim AL\leq m$. Indeed, if some vectors are linearly dependent, then their images under $A$ are linearly dependent as well.
\end{proof}

\begin{proposition}\label{1640}
Let $P\subset\R^d$ be a polyhedral set with non-empty interior. Fix some $k\in\{1,\ldots,d\}$ and let  $A:\R^d\to\R^k$ be a linear map of full rank $k$.
\begin{enumerate}
\item[(a)] If  $F$ is a proper face of $AP$,  then $F = AG$ for a proper face $G\in \cF(P)$ with $\dim G\geq \dim F$.
\item[(b)] If, moreover,  $\Ker A$ is in general position with respect to $P$, then
     \begin{align*}
        \dim F=\dim G.
    \end{align*}
    Also, $G$ is unique in the following sense: If $G'\in \cF(P)$ satisfies $AG' = F$, then $G'=G$.
\end{enumerate}
\end{proposition}
\begin{proof}
We prove (a).
By definition of a face, there exists a supporting affine hyperplane $H\subset\R^k$ of the polyhedral set $AP$ such that
\begin{align*}
        F=H\cap AP.
    \end{align*}
Since $A$ has full rank, $A^{-1}H$ is an affine hyperplane in $\R^d$. Moreover, we claim that  $A^{-1}H$ is a supporting hyperplane of $P$. Indeed, if $H=\{y\in \R^k: \phi(y) = c\}$ and $AP\subset \{y\in \R^k: \phi(y) \geq c\}$ for some linear functional $\phi:\R^k\to\R$ (that does not vanish identically) and some constant $c\in\R$, then $A^{-1}H = \{x\in \R^d: \phi(Ax) = c\}$ and $P\subset \{x\in\R^d: \phi(Ax) \geq c\}$. To verify the last claim, assume that some $x\in P$ satisfies $\phi(Ax) <c$. But then $y:=Ax\in AP$ and $\phi(y) < c$, a contradiction. Since $x\mapsto \phi(Ax)$ is a linear functional on $\R^d$ that does not vanish identically, it follows that $A^{-1} H$ is a supporting hyperplane of $P$. It follows that
\begin{align*}
G:=(A^{-1}F)\cap P = A^{-1}(H\cap AP) \cap P = A^{-1}H\cap A^{-1}A P \cap P = A^{-1} H \cap P
\end{align*}
is a face of $P$. Let us finally check that $A G  = F$.  Clearly, $A G = A (A^{-1}F \cap P) \subset F$.  To prove the converse inclusion $F\subset AG$, take some $f\in F = H\cap AP$.
It follows that there is $p\in P$ with $f = Ap$.   Suppose that $p\notin A^{-1}H$, then $\phi(Ap)>c$. Hence, $\phi(f) > c$, which is a contradiction to $f\in F \subset  H$. We just proved that $f=Ap$ with $p\in A^{-1} H \cap P = G$. Hence, $F = AG$, which proves (a) because the map $A$ cannot increase dimension and hence $\dim F\leq \dim G$.

Statement~(b) follows directly from Lemma~\ref{1641}. Indeed, since $\dim AP = k$ and $F$ is a proper face of $AP$, we have $\dim F <k$. By Lemma~\ref{1641}, we must have $\dim F = \dim AG = \min (k, \dim G) = \dim G$. To prove the uniqueness of $G$, one can argue as follows. If $AG'=F$ and since $F\subset H$, we must have $G' \subset A^{-1}H$ and thus also $G'\subset (A^{-1}H) \cap P = G$. If $G'$ would be a proper subset of $G$, it would have a strictly smaller dimension than $\dim G$, therefore also the dimension of $F= AG'$ would be strictly smaller than $\dim G = \dim F$, which is a contradiction.
\end{proof}

\begin{proposition}\label{1122}
Let $P\subset\R^d$ be a  polyhedral set with non-empty interior. Fix some integer $0\leq j < k \leq d$. Let  $A:\R^d\to\R^k$ be a linear map of full rank $k$ such that $\Ker A$ is in general position with respect to $P$. If $F$ is a $j$-face of $P$, then  the following statements are equivalent:
\begin{enumerate}
    \item[(a)] $AF$ is a face of $AP$;
    \item[(b)] $AF$ is a $j$-face of $AP$;
    \item[(c)] $AF\cap\Int AP=\varnothing$.
    \item[(d)] $A(T_F(P)) \neq  \R^k$.
    \item[(e)] $(\Int T_F(P)) \cap \Ker A =\varnothing$.
    \item[(f)] $T_F(P) \cap \Ker A = \{0\}$.
\end{enumerate}
\end{proposition}
Before we can start with the proof we need to state some lemmas. 
\begin{lemma}\label{1146}
Let $k\in \{1,\ldots,d-1\}$. Consider some  convex cone  $C\subset \R^d$ with non-empty interior and a linear map $A:\R^{d} \to\R^k$ of full rank $k$. The following two conditions are equivalent:
\begin{enumerate}
    \item[(a)] $\Int C\cap \Ker A\ne\varnothing$;
    \item[(b)] $AC = \R^k$.
\end{enumerate}
\end{lemma}
\begin{proof}
See~\cite[Lemma~5.1]{goetze_kabluchko_zaporozhets}, where we have $\lin (AC) = \R^k$ because $\Int C \neq \varnothing$ and $A$, being a linear surjection,  maps open sets to open sets.
\end{proof}

\begin{lemma}\label{1134}
For  any   set  $M\subset \R^k$ the following two conditions are equivalent:
\begin{enumerate}
    \item[(a)] $0\in\relint \conv M$;
    \item[(b)] $\pos  M=\lin M$.
\end{enumerate}
\end{lemma}
\begin{proof}
See~\cite[Proposition~5.2]{goetze_kabluchko_zaporozhets}.
\end{proof}

\begin{lemma}\label{lem:intersects_boundary_intersects_interior}
Let $C\subset \R^d$ be a polyhedral cone of full dimension $\dim C=d$.
Let  a proper linear subspace $S\subset \R^d$ be in general position with respect to the set of the linear hulls of its faces $\{\lin (F): F\in \cF(C) \}$.
If $S$ intersects $C$, then it also intersects its interior $\Int C$, i.e.
		$$
		S\cap C \neq \varnothing \Rightarrow S\cap \Int C \neq \varnothing.
		$$
\end{lemma}
\begin{proof}
See the proof of Lemma~3.5 in~\cite{KVZ15} or~\cite[Lemma~5.7]{kabluchko_seidel} (which is stated for polytopes but is true for arbitrary polyhedral sets).
\end{proof}

\begin{proof}[Proof of Proposition~\ref{1122}]
Note that $\dim AP = k$.  It follows from Lemma~\ref{1641} that (a) implies (b), and obviously (b) implies (c).

Let us prove that (c) implies (d). Take an arbitrary $f\in\relint F$. Then, by (c) we have $Af \notin \Int AP$ and hence $0\notin\Int A(P-f)$. Therefore, making use of Lemma~\ref{1134} and taking into account that the set $A(P-f)$ is convex and has non-empty interior, we have $\pos  A(P-f) \neq \R^k$. But then
\[
A (T_F(P)) = A\pos  (P-f)=\pos  A(P-f)\neq \R^k,
\]
thus proving (d).

The equivalence of (d), (e) and (f) is stated in Lemmas~\ref{1146} and~\ref{lem:intersects_boundary_intersects_interior}.  It remains to prove that (f) implies (a). So, let
\begin{equation}\label{eq:T_F_P_Ker_A}
T_F(P) \cap \Ker A = \{0\}.
\end{equation}
Since $A$ preserves convexity, $AF$ is a convex subset of $AP$. To prove that $AF$ is a face of $AP$ it suffices to prove the following statement (which is, in fact, a definition of a face; see~\cite[p.~18]{schneider_book_brunn_mink}): If $x = Af\in AF$ can be represented as $x= \frac 12 (x_1+x_2)$ for some $x_1,x_2\in AP$, then $x_1,x_2\in AF$.  Write $x_1= Ap_1$ and $x_2 = Ap_2$ for some $p_1,p_2\in P$.   Then, $x = Ap$ with $p := \frac 12(p_1+p_2)\in P$.
So, $x=Ap =Af$ with $f\in F$ and $p\in P$. We claim that this implies that $p=f$. This can be verified as follows. On the one hand, we have $p-f\in \Ker A$ because $A(p-f) = Ap - Af = x-x = 0$. On the other hand, we have $p-f = (p-f_0) +(f_0-f) \in T_F(P)$, where $f_0\in \relint F$ is arbitrary and we have used that $p-f_0 \in T_F(P)$ be the definition of the tangent cone and that $f_0 - f$ belongs to $f_0-\aff F$, which is the lineality space of $T_F(P)$. To summarize, $p-f \in T_F(P) \cap \Ker A$, hence $p=f$ by~\eqref{eq:T_F_P_Ker_A}.

So, $p=f\in F$. But since $F$ is a face of $P$ and $p=\frac 12 (p_1+p_2)$, we must have $p_1\in F$ and $p_2 \in F$. It follows that $x_1= Ap_1\in AF$ and $x_2= Ap_2\in AF$, thus proving the claim.
\end{proof}

\section{Proofs of Theorems~\ref{820},~\ref{623} and~\ref{theorem:grassmann_sums_proj_polytope}} \label{2157}

\begin{proof}[Proof of Theorem~\ref{820}]
By the definition of the Grassmann angles, see~\eqref{1138}, we have
\begin{align*}
\gamma_k(T_F(P)) = \P[T_F(P) \cap W_{d-k}\ne \{0\}].
\end{align*}
Applying Proposition~\ref{1122} (in particular, the equivalence between (a) and (f)) in the setting when $A$ is the  orthogonal projection $\Pi_{W_{d-k}^\perp}$ on $W_{d-k}^\bot$ (which we identify with $\R^{k}$), we arrive at
\begin{align*}
\gamma_k(T_F(P))
=
\P\big[\Pi_{W_{d-k}^\perp} F \not\in \cF(\Pi_{W_{d-k}^\perp}P)\big].
\end{align*}
Note that the general position assumption of Proposition~\ref{1122} is fulfilled with probability $1$ for the random linear subspace $\Ker A = W_{d-k}$; see \cite[Lemma 13.2.1]{SW08}.
Observing that  $W_{d-k}^\perp$ has the same distribution as $W_k$, we arrive at
$$
\gamma_k(T_F(P))
=
\P\big[\Pi_{W_{k}} F \not\in \cF(\Pi_{W_{k}}P)\big]
=
\P\big[\Pi_{k} F \not\in \cF(\Pi_{k}P)\big].
$$
To show that $\cF(\Pi_{k}P)$ can be replaced by $\cF_j(\Pi_{k}P)$ on the right-hand side, we can use the same argument as above, but this time appeal to the equivalence of (b) and (f) in Proposition~\ref{1122}.
\end{proof}

\begin{proof}[Proof of Theorem~\ref{623}]
Taking the sum of~\eqref{eq:gamma_k_proof} over all $F\in \cF_j(P)$, we obtain
$$
\sum_{F\in \cF_j(P)} \gamma_k(T_F(P))
=
\sum_{F\in \cF_j(P)} \P\big[\Pi_{k} F \not\in \cF_j(\Pi_{k}P)\big]
=
f_j(P) - \sum_{G\in \cF_j(P)} \P\big[\Pi_{k} G \in \cF_j(\Pi_{k}P)\big].
$$
Writing the probabilities as the expectations of the corresponding indicator functions, we can rewrite this as
$$
\sum_{F\in \cF_j(P)} \gamma_k(T_F(P))
=
f_j(P) - \E \sum_{G\in \cF_j(P)} \ind\{\Pi_{k} G \in \cF_j(\Pi_{k}P)\}.
$$
According to Proposition~\ref{1640} every $j$-face of $\Pi_{k}P$ is of the form $\Pi_k G$ for some unique $G\in \cF_j(P)$. Thus, the sum on the right-hand side equals  $f_j(\Pi_k P)$ and  we arrive at
$$
\sum_{F\in \cF_j(P)} \gamma_k(T_F(P))
=
f_j(P) - \E f_j(\Pi_k P),
$$
thus proving the claim.
\end{proof}

\begin{proof}[Proof of Theorem~\ref{theorem:grassmann_sums_proj_polytope}]
Take $d,n\in\N$ satisfying $n\ge d$. Furthermore, let $A\in\R^{d\times n}$ be a Gaussian random matrix, let $P\subset \R^n$ be a polyhedral set with non-empty interior and let $j\in\{0,\dots,d-1\}$ and $k\in\{j-1,\dots,d-1\}$ be given. At first, we are going to show that $\ker A$ is in general position with respect to $P$ a.s.~and that $\rank A=d$ holds with probability $1$.
In order to prove that $\rank  A=d$ a.s., we will show that the row-vectors $\xi_1,\dots,\xi_d$ of $A$ are linearly dependent with probability $0$. Note that $\xi_1,\dots,\xi_d$ are independent and $n$-dimensional standard Gaussian distributed. Then, we have
\begin{align*}
&\P[\xi_1,\dots,\xi_d\text{ are linearly dependent}]\\
&	\quad\le d\cdot \P[\xi_1\in\lin\{\xi_2,\dots,\xi_d\}]\\
&	\quad=d\cdot \int_{(\R^n)^{d-1}}\P(\xi_1\in\lin\{x_2,\dots,x_d\}|\xi_2=x_2,\dots,\xi_d=x_d)\P_{(\xi_2,\dots,\xi_d)}(\text{d}(x_2,\dots,x_d))\\
&	\quad=d\cdot \int_{(\R^n)^{d-1}}\P(\xi_1\in\lin\{x_2,\dots,x_d\})\P_{(\xi_2,\dots,\xi_d)}(\text{d}(x_2,\dots,x_d))\\
&	\quad=0,
\end{align*}
since $\dim\lin\{x_2,\dots,x_d\}\le d-1<n$. Note that $\P_{(\xi_2,\dots,\xi_d)}$ denotes the joint Gaussian probability law of $(\xi_2,\dots,\xi_d)$. Thus, $\rank A=d$ holds true with probability 1.
Furthermore, $\ker A$ is in general position to every subspace $L\subset\R^n$, following~\cite[Lemma 13.2.1]{SW08}, since $\ker A$ is invariant under rotations and therefore has the uniform distribution on the Grassmannian of all $(n-d)$-dimensional subspaces.

Assume first that $k\in\{j,\dots,d-1\}$, meaning that the case $k=j-1$ is postponed.  Using Proposition~\ref{1640}, we obtain
\begin{align*}
\E\sum_{F\in\cF_j(AP)}\gamma_k(T_F(AP))
&	=\E\sum_{G\in\cF_j(P)}\gamma_k(T_{AG}(AP))\ind_{\{AG\in\cF_j(AP)\}}.
\end{align*}
We claim that $AT_G(P)=T_{AG}(AP)$ holds for a face $G\in\cF_j(P)$ provided its projection is also a face $AG\in\cF_j(AP)$. In order to prove this, note that the tangent cone $T_G(P)$ can equivalently be defined as $\pos (P-g)$ for any point $g\in\relint G$. Since $A$ is a linear mapping, we obtain
\begin{align*}
AT_G(P)=A\pos(P-g)=\pos(A(P-g))=\pos(AP-Ag).
\end{align*}
It is left to show that $Ag$ lies in the relative interior of $AG$. Obviously, we know that $Ag\in AG$. Now suppose $Ag\notin \relint AG$, that is, it lies in a face of $AP$ of dimension smaller than $j$ (which is the dimension of $AG$). Due to Proposition~\ref{1640}, this implies that $g$ lies in a face of $P$ of smaller dimension than $j$ which is a contradiction to $g\in\relint G$. Thus, $Ag\in\relint AG$ and we have $AT_G(P)=\pos(AP-Ag)=T_{AG}(AP)$, which yields
\begin{align*}
\E\sum_{F\in\cF_j(AP)}\gamma_k(T_F(AP))
&	=\E\sum_{G\in\cF_j(P)}\gamma_k(AT_{G}(P))\ind_{\{AG\in\cF_j(AP)\}}.
\end{align*}

Proposition~\ref{1122} implies that $AG\in\cF_j(AP)$ is equivalent to $AT_G(P)\neq\R^d$ since $\rank A=d$ and $\ker A$ is in general position with respect to $P$ a.s. Therefore, we arrive at
\begin{align}\label{eq:123}
\E\sum_{F\in\cF_j(AP)}\gamma_k(T_F(AP))
&=	\E\sum_{G\in\cF_j(P)}\gamma_k(AT_{G}(P))\ind_{\{AT_G(P)\neq\R^d\}}\notag\\
&=	\sum_{G\in\cF_j(P)}\E\big[\gamma_k(AT_{G}(P))\ind_{\{AT_G(P)\neq\R^d\}}\big].
\end{align}
From \cite[Corollary~3.6]{goetze_kabluchko_zaporozhets}, we obtain that
\begin{align*}
\E\big[\gamma_k(AT_{G}(P))\ind_{\{AT_G(P)\neq\R^d\}}\big]=\gamma_k(T_G(P))-\gamma_d(T_G(P)).
\end{align*}
Note that we used that $T_G(P)$ is not a linear subspace because $P$ is not a linear subspace (otherwise we would have $P=\R^n$ since $\Int P\neq \varnothing$, and the statement of the theorem would become empty).  Applying this to~\eqref{eq:123} yields
\begin{align*}
\E\sum_{F\in\cF_j(AP)}\gamma_k(T_F(AP))=\sum_{G\in\cF_j(P)}\big(\gamma_k(T_G(P))-\gamma_d(T_G(P))\big),
\end{align*}
which is the required formula.
To complete the proof, note that in the case $k=j-1$ the required formula reduces (using that $\gamma_{-1}(C)  = 1$) to the identity
$$
\E f_j(AP) = f_j(P) - \sum_{G\in \cF_j(P)} \gamma_d(T_G(P)).
$$
This identity is verified by Theorem~\ref{623}. The fact that the orthogonal projection  can be replaced by the Gaussian projection follows from the argument of~\cite{BV94}.
\end{proof}

\bibliographystyle{plainnat}
\bibliography{bib}

\end{document}